\documentclass[a4paper,10pt, oneside]{article}
\usepackage{amsmath, amssymb, amsthm,graphicx}
\usepackage[font=small,labelfont=md,textfont=it]{caption}
\usepackage{floatrow,float}
\usepackage[titletoc, title]{appendix}
\usepackage[colorlinks,linkcolor=blue,citecolor=blue]{hyperref}
\usepackage{longtable}
\usepackage{pgfplots}
\usepackage{diagbox}
\usepackage{booktabs,makecell,multirow}
\usepackage[capitalise, nosort]{cleveref}
\usepackage{cases,color}
\crefname{equation}{}{}
\crefname{lem}{Lemma}{Lemmas}
\crefname{thm}{Theorem}{Theorems}

\DeclareMathOperator{\D}{D}
\DeclareMathOperator{\I}{I}

\newcommand{\dual}[1]{\left\langle {#1} \right\rangle}

\newcommand{\nm}[1]{\left\lVert {#1} \right\rVert}
\newcommand{\snm}[1]{\left\lvert {#1} \right\rvert}
\newcommand{\ssnm}[1]
{
  \left\vert\kern-0.25ex
  \left\vert\kern-0.25ex
  \left\vert
  {#1}
  \right\vert\kern-0.25ex
  \right\vert\kern-0.25ex
  \right\vert
}

\makeatletter
\def\spher@harm#1{%
  \vbox{\hbox{%
    \offinterlineskip
    \valign{&\hb@xt@2\p@{\hss$##$\hss}\vskip.2ex\cr#1\crcr}%
  }\vskip-.36ex}%
}
\def\gshone{\spher@harm{.}}
\def\gshtwo{\spher@harm{.&.}}
\def\gshthree{\spher@harm{.&.&.}}
\let\gsh\spher@harm
\makeatother

\newtheorem{lem}{Lemma}[section]
\newtheorem{rem}{Remark}[section]
\newtheorem{thm}{Theorem}[section]

\makeatletter\def\@captype{table}\makeatother

\begin{document}

%\title{
%  \Large\bf Analysis of a time-stepping scheme for time fractional diffusion
%problems with nonsmooth data } \author{ {School of Mathematics, Sichuan
%University, Chengdu 610064, China} }
\title{
	%\Large \bf A Petrov-Galerkin time stepping method for fractional wave problem
	\Large \bf  Analysis of a time-stepping scheme for time fractional diffusion
	problems with nonsmooth data 
	\thanks
	{
		This work was supported in part by National Natural Science Foundation
		of China (11771312).
	}
}
\author{
	Binjie Li \thanks{Email: libinjie@scu.edu.cn},
	Hao Luo \thanks{Corresponding author. Email: galeolev@foxmail.com},
	Xiaoping Xie \thanks{Email: xpxie@scu.edu.cn} \\
	{School of Mathematics, Sichuan University, Chengdu 610064, China}
}

\date{}
\maketitle

\begin{abstract}
  This paper establishes the convergence of a time-steeping scheme for  time
  fractional diffusion problems with nonsmooth data. We first analyze the
  regularity of the model problem with nonsmooth data, and
  then prove that the time-steeping   scheme possesses optimal convergence rates in $
  L^2(0,T;L^2(\Omega)) $-norm and $ L^2(0,T;H_0^1(\Omega)) $-norm with respect
  to the regularity of the solution.  Finally,   numerical results are
  provided to verify the theoretical results.

  %More importantly,
  %This method uses the space of discontinuous functions that are piecewise
  %constant in the temporal discretization, and uses the space of continuous
  %functions that are piecewise linear in the spatial discretization.

\end{abstract}


\section{Introduction}
This paper considers the following time fractional diffusion problem:
\begin{equation}
\label{eq:model}
\left\{
\begin{aligned}
\D_{0+}^\alpha (u-u_0) - \Delta u & = f   &  & \text{in $ \Omega \times (0,T) $,}         \\
u                                 & = 0   &  & \text{on $ \partial\Omega \times (0,T) $,} \\
u(0)                              & = u_0 &  & \text{ in $ \Omega $, }
\end{aligned}
\right.
\end{equation}
where $ 0 < \alpha < 1 $, $ \D_{0+}^\alpha $ is a Riemann-Liouville fractional
differential operator, $ \Omega \subset \mathbb R^d $ ($d=1,2,3$) is a convex
polygonal domain, and $ u_0 $ and $ f $ are given functions.

A considerable amount of numerical algorithms for time fractional diffusion
problems have been developed. Generally, these numerical algorithms can be
divided into three types. The first type uses finite difference
methods to approximate the time fractional derivatives. Despite their ease of
implementation, the fractional difference methods are generally of temporal
accuracy orders no greater than two; see \cite{Yuste2005, Langlands2005,
	Yuste2006, Chen2007, Lin2007, Zhuang2008, Deng2009, Zhang2009524,
	Chen2010,Cui2009, Liu2011,Gao2011,Zeng2013, Wang2014, gao2014new,
	Li2016} and the references therein. The second type applies   spectral
methods to discretize the time fractional derivatives; see \cite{Li2009,
	Zheng2015,Li2017spectral,Zayernouri2014Fractional,Zayernouri2012Karniadakis,Zayernouri2014Exponentially,yang2016spectral,Li2017A}. The main advantage of these algorithms is that they possess
high-order accuracy, provided the solution is sufficiently smooth. The third
type adopts   finite element methods to approximate the time fractional
derivatives; see \cite{Mclean2009Convergence, Mustapha2011Piecewise,
Mustapha2015Time,Li2017A,Mustapha2009Discontinuous,Mustapha2012Uniform,Mustapha2012Superconvergence,Mclean2015Time,mustapha2014well-posedness,Mustapha2014A}. These algorithms are easy to implement, like those in the first type, and possess high-order accuracy.

The convergence analysis of the aforementioned algorithms is generally carried
out on the condition that the underlying solution is sufficiently smooth. So
far, the works on the numerical analysis for nonsmooth data are very limited.
By using the Laplace transformation, Mclean and Thom\'ee \cite{Thomee2010IMA}
analyzed three fully discretizations for fractional order evolution equations, where the
initial values are allowed to have only $ L^2(\Omega) $-regularity. By using a growth
estimation of the Mittag-Leffler function, Jin et al.~\cite{Jin2013SIAM,Jin2015IMA}
analyzed the convergence of a spatial semi-discretization of problem \cref{eq:model}. They derived the following results: if $
f = 0 $, then
\[
\nm{u(t) - u_h(t)}_{L^2(\Omega)} +
h \nm{u(t)-u_h(t)}_{H_0^1(\Omega)} \leqslant
C h^2 \snm{\ln h} t^{-\alpha} \nm{u_0}_{L^2(\Omega)};
\]
if $ u_0 = 0 $ and $ 0 \leqslant \beta < 1 $, then
\[
%\label{eq:jin}
\nm{u\!-\!u_h}_{L^2(0,T;L^2(\Omega))} +
h \nm{u\!-\!u_h}_{L^2(0,T;H_0^1(\Omega))} \leqslant
C h^{2-\beta} \nm{f}_{L^2(0,T; H^{-\beta}(\Omega))},
\]
\[
\nm{u(t)\!-\!u_h(t)}_{L^2(\Omega)}\!+\!
h \nm{u(t)\!-\!u_h(t)}_{H_0^1(\Omega)}\!\leqslant\!
C h^{2-\beta} \snm{\ln h}^2 \nm{f}_{L^\infty(0,t;H^{-\beta}(\Omega))}.
\]
%It is well known that \cref{eq:model} can be written into the following form
%\begin{equation}
%\label{eq:model*}
%\left\{
%\begin{aligned}
%u'+\D_{0+}^{1-\alpha} \Delta u & = \D_{0+}^{1-\alpha}f   &  & \text{in $ \Omega \times (0,T) $,}         \\
%u                                 & = 0   &  & \text{on $ \partial\Omega \times (0,T) $,} \\
%u(0)                              & = u_0 &  & \text{ in $ \Omega $.}
%\end{aligned}
%\right.
%\end{equation}
Recently, McLean and Mustapha
\cite{Mclean2015Time} derived that
\[
\nm{u(t_n) - U^n}_{L^2(\Omega)} \leqslant
C t_n^{-1} \Delta t \nm{u_0}_{L^2(\Omega)}
\]
for a piecewise constant DG scheme in temporal semi-discretization of fractional diffusion problems with $ f = 0$. For more related work, we refer the reader to \cite{Cuesta2006,Mclean2010Thom}.

In this paper, we present a rigorous analysis of the convergence of a
time-stepping scheme for problem \cref{eq:model}, which uses a space of continuous 
piecewise linear functions in the spatial discretization and a space
of piecewise constant functions in the temporal discretization. We first apply the
Galerkin method to investigate the regularity of problem \cref{eq:model} with
non-smooth $ u_0 $ and $ f $, and then we derive the following error estimates:
if $ 0 < \alpha < 1/2 $ and $ 0 \leqslant \beta < 1 $, then
\begin{align*}
& (h+\tau^{\alpha/2})^{-1} \nm{u-U}_{L^2(0,T;L^2(\Omega))} +
\nm{u-U}_{L^2(0,T; H_0^1(\Omega))} \\
\leqslant{} &
C \left( h^{1-\beta} + \tau^{\alpha(1-\beta)/2} \right)
\left(
\nm{f}_{L^2(0,T;H^{-\beta}(\Omega))} +
\nm{u_0}_{H^{-\beta}(\Omega)}
\right);
\end{align*}
if $ 1/2 \leqslant \alpha < 1 $ and $ 2-1/\alpha < \beta < 1 $, then
\begin{align*}
& (h+\tau^{\alpha/2})^{-1} \nm{u-U}_{L^2(0,T;L^2(\Omega))} +
\nm{u-U}_{L^2(0,T; H_0^1(\Omega))} \\
\leqslant{} &
C \left( h^{1-\beta} + \tau^{\alpha(1-\beta)/2} \right)
\left(
\nm{f}_{L^2(0,T;H^{-\beta}(\Omega))} +
\nm{u_0}_{L^2(\Omega)}
\right);
\end{align*}
if $ 1/2 \leqslant \alpha < 1 $ and $ u_0 = 0 $, then the above estimate also holds for
all $ 0 \leqslant \beta < 1 $. Furthermore, if $ 1/2 < \alpha < 1 $ and $ u_0 = 0 $, then
we derive the optimal error estimate
\[
\nm{u-U}_{L^2(0,T;L^2(\Omega))} \leqslant C (h^2+\tau)
\nm{f}_{H^{1-\alpha}(0,T;L^2(\Omega))}.
\]
By the techniques used in our analysis, we can also derive the error estimates
under other conditions; for instance, $ u_0 $ and $ f $ are smoother than the
aforementioned cases.

The rest of this paper is organized as follows. \cref{sec:pre} introduces some
Sobolev spaces, the Riemann-Liouville fractional calculus operators, the weak
solution to problem \cref{eq:model}, and a time-stepping scheme.
\cref{sec:regu} investigates the regularity of the weak solution, and
\cref{sec:conv} establishes the convergence of the time-steeping scheme.
Finally, \cref{sec:numer} provides some numerical experiments to verify the
theoretical results.

\section{Preliminaries}
\label{sec:pre}
\textit{\textbf{Sobolev Spaces.}} For a Lebesgue measurable subset $ \omega $ of
$ \mathbb R^l $ ($l=1,2,3$), we use $ H^\gamma(\omega) $ ($ -\infty < \gamma <
\infty $) and $ H_0^\gamma(\omega) $ ($ 0<\gamma<\infty $) to denote two
standard Sobolev spaces \cite{Tartar2007}. Let $ X $ be a separable Hilbert
space with an inner product $ (\cdot,\cdot)_X $ and an orthonormal basis $
\{e_i: i \in \mathbb N\} $. We use $ H^\gamma(0,T;X) $ ($0 \leqslant \gamma <
\infty $) to denote an usual vector valued Sobolev space, and for $ 0 < \gamma <
1/2 $, we also use the norm
\[
  \snm{v}_{H^\gamma(0,T;X)} := \left(
    \sum_{i=0}^\infty \snm{(v,e_i)_X}_{H^\gamma(0,T)}^2
  \right)^{1/2}, \quad \forall v \in H^\gamma(0,T;X).
\]
Here, the norm $ \snm{\cdot}_{H^\gamma(0,T)} $ is given by
%Here, if $ w \in H^\gamma(0,T) $ ($ 0<\gamma<1/2 $) , then
\[
  \snm{w}_{H^\gamma(0,T)} := \left(
    \int_\mathbb R \snm{\xi}^{2\gamma}
    \snm{\mathcal F(w\chi_{(0,T)})(\xi)}^2 \,\mathrm{d}\xi
  \right)^{1/2}, \quad \forall w \in H^\gamma(0,T),
\]
where $ \mathcal F: L^2(\mathbb R) \to L^2(\mathbb R) $ is the Fourier transform
operator and $ \chi_{(0,T)} $ is the indicator function of $ (0,T) $.

\medskip\noindent\textit{\textbf{Fractional Calculus Operators.}} Let $ X $ be a
Banach space and let $ -\infty \leqslant a < b \leqslant \infty $. For $ 0 <
\gamma < \infty $, define
\begin{align*}
  \left(\I_{a+}^\gamma v\right)(t) &:=
  \frac1{ \Gamma(\gamma) }
  \int_a^t (t-s)^{\gamma-1} v(s) \, \mathrm{d}s, \quad t\in(a,b), \\
  \left(\I_{b-}^\gamma v\right)(t) &:=
  \frac1{ \Gamma(\gamma) }
  \int_t^b (s-t)^{\gamma-1} v(s) \, \mathrm{d}s, \quad t\in(a,b),
\end{align*}
for all $ v \in L^1(a,b;X) $, where $ \Gamma(\cdot) $ is the gamma function.
For $ j-1 < \gamma < j $ with $ j \in \mathbb N_{>0} $, define
\begin{align*}
  \D_{a+}^\gamma & := \D^j \I_{a+}^{j-\gamma}, \\
  \D_{b-}^\gamma & := (-1)^j \D^j \I_{b-}^{j-\gamma},
\end{align*}
where $ \D $ is the first-order differential operator in the distribution sense.
%\noindent When $ \D_{a+}^\alpha $ (or $\D_{b-}^\alpha$) arises, $ b $ (or $a$)
%is determined by the context.

\medskip\noindent\textit{\textbf{Eigenvectors of $-\Delta$.}} It is well known
that there exists an orthonormal basis
\[
  \{\phi_i: i \in \mathbb N \} \subset H_0^1(\Omega) \cap H^2(\Omega)
\]
of $ L^2(\Omega) $ such that
\[
  -\Delta \phi_i = \lambda_i \phi_i,
\]
where $ \{ \lambda_i: i \in \mathbb N \} \subset \mathbb R_{>0} $ is a
non-decreasing sequence. For any $ 0 \leqslant \gamma < \infty $, define
\[
  \dot H^\gamma(\Omega) := \left\{
    v \in L^2(\Omega):\
    \sum_{i=0}^\infty \lambda_i^\gamma (v,\phi_i)_{L^2(\Omega)}^2
    < \infty
  \right\},
\]
and equip this space with the norm
\[
  \nm{\cdot}_{\dot H^\gamma(\Omega)} :=
  \left(
    \sum_{i=0}^\infty \lambda_i^\gamma (\cdot, \phi_i)_{L^2(\Omega)}^2
  \right)^{1/2}.
\]
For $ \gamma \in [0,1] \setminus \{0.5\} $, the space $ \dot H^\gamma(\Omega) $
coincides with $ H_0^\gamma(\Omega) $ with equivalent norms, and for $ 1 <
\gamma \leqslant 2 $, the space $ \dot H^\gamma(\Omega) $ is continuously
embedded into $ H^\gamma(\Omega) $.

\medskip\noindent\textit{\textbf{Weak Solution.}} Define
\[
  W := H^{\alpha/2}(0,T;L^2(\Omega)) \cap L^2(0,T;\dot H^1(\Omega))
\]
and endow this space with the norm
\[
  \nm{\cdot}_{W} :=
  \snm{\cdot}_{H^{\alpha/2}(0,T;L^2(\Omega))} +
  \nm{\cdot}_{L^2(0,T;\dot H^1(\Omega))}.
\]
Assuming that
\begin{equation}
  \label{eq:basic_assum}
  \D_{0+}^\alpha u_0 + f \in W^*,
\end{equation}
we call $ u \in W $ a weak solution to problem \cref{eq:model} if
\begin{equation}
  \label{eq:weak_form}
  \dual{ \D_{0+}^\alpha u, v }_{ H^{\alpha/2}( 0,T;L^2(\Omega) ) } +
  \dual{ \nabla u, \nabla v }_{ \Omega \times (0,T) } =
  \dual{\D_{0+}^\alpha u_0 + f,v}_{W}
\end{equation}
for all $ v \in W $. Throughout the paper, if $ \omega $ is a Lebesgue
measurable set of $ \mathbb R^l $ ($ l= 1,2,3,4 $) then the symbol $
\dual{p,q}_\omega $ means $ \int_\omega pq $, and if $ X $ is a Banach space
then $ \dual{\cdot,\cdot}_X $ means the duality pairing between $ X^* $ (the
dual space of $ X $) and $ X $.
\begin{rem}
  \label{rem:well-posedness}
  The above weak solution is first introduced by Li and Xu \cite{Li2009}.
  Evidently, the well-known Lax-Milgram theorem indicates that problem
  \cref{eq:model} admits a unique weak solution by \cref{lem:coer}. Moreover,
  \[
    \nm{u}_{W} \leqslant C \nm{\D_{0+}^\alpha u_0 + f}_{W^*},
  \]
  where $ C $ is a positive constant that depends only on $ \alpha $.
\end{rem}

\medskip\noindent\textit{\textbf{Discretization.}} Let
\[
  0 = t_0 < t_1 < \ldots < t_J = T
\]
be a partition of $ [0,T] $. Set $ I_j := (t_{j-1},t_j) $ for each $ 1 \leqslant j
\leqslant J $, and we use $ \tau $ to denote the maximum length of these intervals. Let $
\mathcal K_h $ be a conventional conforming and shape regular triangulation of $ \Omega $
consisting of $ d $-simplexes, and we use $ h $ to denote the maximum diameter of the
elements in $ \mathcal K_h $. Define
\begin{align*}
  \mathcal S_h&:= \left\{
    v_h \in H_0^1(\Omega):\
    v_h|_K \in P_1(K),\ \forall \,K \in \mathcal K_h
  \right\},\\
  \mathcal M_{h,\tau} &:= \left\{
    V \in L^2(0,T;\mathcal S_h):\
    V|_{I_j} \in P_0(I_j;\mathcal S_h),\ \forall\,1\leqslant j\leqslant J
  \right\},
\end{align*}
where $ P_1(K) $ is the set of all linear polynomials defined on $ K $, and $
P_0(I_j; \mathcal S_h) $ is the set of all constant $ \mathcal S_h$-valued
functions defined on $ I_j $.

Naturally, the discretization of problem \cref{eq:weak_form} reads as follows:
seek $ U \in \mathcal M_{h,\tau} $ such that
\begin{equation}
  \label{eq:algor}
  \dual{\D_{0+}^\alpha U,V}_{H^{\alpha/2}(0,T;L^2(\Omega))} +
  \dual{\nabla U,\nabla V}_{\Omega \times (0,T)} =
  \dual{\D_{0+}^\alpha u_0 + f,V}_{W}
\end{equation}
for all $ V \in \mathcal M_{h,\tau} $.

\begin{rem}
  Similarly to the stability estimate in \cref{rem:well-posedness}, we have
  \[
    \nm{U}_W \leqslant C \nm{\D_{0+}^\alpha u_0 + f}_{W^*},
  \]
  where $ C $ is a positive constant depending only on $ \alpha $. Therefore,
  problem \cref{eq:algor} is also stable under condition \cref{eq:basic_assum}.
\end{rem}
%\begin{rem}
  %It is easy to verify that
  %\begin{align*}
    %& L^2(0,T;H^{-1}(\Omega)) \subset W^*, \\
    %& L_{\alpha}^2(0,T;L^2(\Omega)) \subset H^{-\alpha/2}(0,T;L^2(\Omega))
    %\subset W^*.
  %\end{align*}
%\end{rem}

\section{Regularity}
\label{sec:regu}
Let us first consider the following problem: seek $ y \in H^{\alpha/2}(0,T) $
such that
\begin{equation}
  \label{eq:frac_ode_weak}
  \dual{ \D_{0+}^\alpha(y-y_0), z }_{ H^{\alpha/2}(0,T) } +
  \lambda \dual{y,z}_{(0,T)} = \dual{g,z}_{(0,T)}
\end{equation}
for all $ z \in H^{\alpha/2}(0,T) $, where $ g \in L^2(0,T) $, and $ y_0 $ and $
\lambda>1 $ are two real constants. By \cref{lem:coer}, the Lax-Milgram theorem
indicates that the above problem admits a unique solution $ y \in
H^{\alpha/2}(0,T) $. Moreover, it is evident that
\begin{equation}
  \label{eq:frac_ode}
  \D_{0+}^\alpha (y-y_0) = g -\lambda y
\end{equation}
in $ L^2(0,T) $.

For convenience, we use the following convention: if the symbol $ C $ has
subscript(s), then it means a positive constant that depends only on its
subscript(s), and its value may differ at each of its occurrence(s).
Additionally, in this section we assume that $ u $ and $ y $ are the solutions
to problems \cref{eq:weak_form,eq:frac_ode_weak}, respectively.

\begin{lem}
  \label{lem:regu_ode_I}
  If $ 0 < \alpha < 1/2 $ and $ 0 \leqslant \beta < 1 $, then
  \begin{equation}
    \label{eq:regu_ode_I}
    \begin{aligned}
      & \lambda^{\beta/2}
      \snm{y}_{H^{\alpha(1-\beta/2)}(0,t)} +
      \lambda^{(1+\beta)/2} \snm{y}_{H^{(1-\beta)\alpha/2}(0,t)} +
      \lambda \nm{y}_{L^2(0,t)} \\
      \leqslant{} &
      C_\alpha \left( \nm{g}_{L^2(0,t)} + t^{1/2-\alpha} \snm{y_0} \right)
    \end{aligned}
  \end{equation}
  for all $ 0 < t < T $.
\end{lem}
\begin{proof}
  Let us first prove that $ y \in H^\alpha(0,T) $. By the definition of $
  \D_{0+}^\alpha $, equality \cref{eq:frac_ode} implies
  \[
    \left( \I_{0+}^{1-\alpha}(y-y_0) \right)' = g - \lambda y,
  \]
  so that using integration by parts gives
  \[
    \I_{0+}^{1-\alpha} (y-y_0) =
    \left( \I_{0+}^{1-\alpha} (y-y_0) \right)(0) +
    \I_{0+} (g-\lambda y).
  \]
  In addition, since
  \[
    \snm{ \left( \I_{0+}^{1-\alpha} (y-y_0) \right)(s) } \leqslant
    \frac1{\Gamma(1-\alpha)} \sqrt{\frac{s^{1-2\alpha}}{1-2\alpha}}
    \, \nm{y-y_0}_{L^2(0,s)}, \quad 0 < s < T,
  \]
  we have
  \[
    \left( \I_{0+}^{1-\alpha} (y-y_0) \right)(0) =
    \lim_{s \to 0+} \left( \I_{0+}^{1-\alpha} (y-y_0) \right)(s) = 0.
  \]
  Consequently,
  \[
    \I_{0+}^{1-\alpha} (y-y_0) = \I_{0+} (g-\lambda y),
  \]
  and hence a simple computation gives that
  \[
    y = y_0 + \I_{0+}^\alpha (g-\lambda y).
  \]
  Therefore, \cref{lem:regu} indicates that $ y \in H^\alpha(0,T) $.

  Then let us prove that
  \begin{equation}
    \label{eq:regu_ode_I-12}
    \snm{y}_{H^\alpha(0,t)}^2 + \lambda \snm{y}_{H^{\alpha/2}(0,t)}^2 +
    \lambda^2 \nm{y}_{L^2(0,t)}^2 \leqslant C_\alpha
    \left(
      \nm{g}_{L^2(0,t)}^2 + t^{1-2\alpha} \snm{y_0}^2
    \right).
  \end{equation}
  Multiplying both sides of \cref{eq:frac_ode} by $ y $ and integrating over $
  (0,t) $ yields
  \[
    \dual{\D_{0+}^\alpha y, y}_{(0,t)} +
    \lambda \nm{y}_{L^2(0,t)}^2 =
    \dual{g,y}_{(0,t)} +
    \dual{\D_{0+}^\alpha y_0, y}_{(0,t)}.
  \]
  Since
  \begin{align*}
    \dual{g,y}_{(0,t)} & \leqslant \frac1\lambda \nm{g}_{L^2(0,t)}^2 +
    \frac\lambda4 \nm{y}_{L^2(0,t)}^2, \\
    \dual{\D_{0+}^\alpha y_0,y}_{(0,t)} & \leqslant
    \frac1\lambda \nm{\D_{0+}^\alpha y_0}_{L^2(0,t)}^2 +
    \frac\lambda4 \nm{y}_{L^2(0,t)}^2,
  \end{align*}
  we have
  \[
    \dual{\D_{0+}^\alpha y,y}_{(0,t)} + \lambda \nm{y}_{L^2(0,t)}^2
    \leqslant C_\alpha \left(
      \lambda^{-1} \nm{g}_{L^2(0,t)}^2 +
      \lambda^{-1} t^{1-2\alpha} \snm{y_0}^2
    \right).
  \]
  From \cref{lem:coer} it follows that
  \[
    %\label{eq:regu_ode_I-1}
    \lambda \snm{y}_{H^{\alpha/2}(0,t)}^2 + \lambda^2 \nm{y}_{L^2(0,t)}^2
    \leqslant C_\alpha \left(
      \nm{g}_{L^2(0,t)}^2 + t^{1-2\alpha} \snm{y_0}^2
    \right).
  \]
  Analogously, multiplying both sides of \cref{eq:frac_ode} by $ \D_{0+}^\alpha
  y $ and integrating over $ (0,t) $, we obtain
  %\[
    %\nm{\D_{0+}^\alpha y}_{L^2(0,t)}^2 +
    %\lambda \dual{\D_{0+}^\alpha y,y}_{(0,t)} =
    %\dual{g, \D_{0+}^\alpha y}_{(0,t)} +
    %\dual{\D_{0+}^\alpha y_0, \D_{0+}^\alpha y}_{(0,t)},
  %\]
  %and then a simple calculation gives
  %\[
    %\nm{\D_{0+}^\alpha y}_{L^2(0,t)}^2 +
    %\lambda \dual{\D_{0+}^\alpha y,y}_{(0,t)} \leqslant
    %C_\alpha \left( \nm{g}_{L^2(0,t)}^2 + t^{1-2\alpha} \snm{y_0}^2 \right).
  %\]
  %By \cref{lem:coer} it follows that
  \[
    \snm{y}_{H^\alpha(0,t)}^2 + \lambda \snm{y}_{H^{\alpha/2}(0,t)}^2
    \leqslant C_\alpha
    \left( \nm{g}_{L^2(0,t)}^2 + t^{1-2\alpha} \snm{y_0}^2 \right).
  \]
  Therefore, combining the above two estimates yields \cref{eq:regu_ode_I-12}.

  Now, let us prove that
  \begin{equation}
    \label{eq:regu_ode_I-34}
    \lambda^\beta \snm{y}_{H^{\alpha(1-\beta/2)}(0,t)}^2 +
    \lambda^{1+\beta} \snm{y}_{H^{\alpha(1-\beta)/2}(0,t)}^2
    \leqslant C_\alpha \left(
      \nm{g}_{L^2(0,t)}^2 + t^{1-2\alpha} \snm{y_0}^2
    \right).
  \end{equation}
  Since
  \[
    \alpha(1-\beta/2) = \beta\, \alpha/2 + (1-\beta)\, \alpha,
  \]
  applying \cite[Proposition~1.32]{Bahouri2011} yields
  \[
    \snm{y}_{H^{\alpha(1-\beta/2)}}(0,t) \leqslant
    \snm{y}_{H^{\alpha/2}(0,t)}^\beta
    \snm{y}_{H^\alpha(0,t)}^{1-\beta}.
  \]
  Therefore, by \cref{eq:regu_ode_I-12} we obtain
  \begin{align*}
    & \lambda^\beta \snm{y}_{H^{\alpha(1-\beta/2)}(0,t)}^2 \leqslant
    \left(\lambda \snm{y}_{H^{\alpha/2}(0,T)}^2 \right)^\beta
    \left( \snm{y}_{H^\alpha(0,t)}^2 \right)^{1-\beta} \\
    \leqslant{} &
    \lambda \snm{y}_{H^{\alpha/2}(0,t)}^2 +
    \snm{y}_{H^\alpha(0,t)}^2 \\
    \leqslant{} &
    C_\alpha \left( \nm{g}_{L^2(0,t)}^2 + t^{1-2\alpha} \snm{y_0}^2 \right),
  \end{align*}
  by the Young's inequality. Analogously, we have
  \[
    \lambda^{1+\beta} \snm{y}_{H^{(1-\beta)\alpha/2}(0,t)}^2 \leqslant
    C_\alpha \left( \nm{g}_{L^2(0,t)}^2 + t^{1-2\alpha} \snm{y_0}^2 \right),
  \]
  and using the above two estimates then proves \cref{eq:regu_ode_I-34}.

  Finally, combing \cref{eq:regu_ode_I-12,eq:regu_ode_I-34} yields
  \cref{eq:regu_ode_I} and thus concludes the proof.
\end{proof}

\begin{lem}
  \label{lem:regu_ode_II}
  If $ 1/2 \leqslant \alpha < 1 $ and $ 0 \leqslant \theta < 1/\alpha -1 $, then
  \begin{small}
  \begin{equation}
    \label{eq:regu_ode_II}
    \begin{aligned}
      & \lambda^{(\theta-1)/2} \nm{y}_{H^\alpha(0,T)} +
      \nm{y}_{H^{\alpha(1+\theta)/2}(0,T)} +
      \lambda^{\theta/2} \snm{y}_{H^{\alpha/2}(0,T)} +
      \lambda^{1/2} \nm{y}_{H^{\alpha\theta/2}(0,T)} \\
      & {} + \lambda^{(1+\theta)/2} \nm{y}_{L^2(0,T)} \leqslant
      C_{\alpha,\theta,T} \left(
        \lambda^{(\theta-1)/2} \nm{g}_{L^2(0,T)} + \snm{y_0}
      \right).
    \end{aligned}
  \end{equation}
  \end{small}
\end{lem}
\begin{proof}
  Proceeding as in the proof of \cref{lem:regu_ode_I} yields
  \[
    y = y_0 + \frac{c}{\Gamma(\alpha)} t^{\alpha-1} +
    \I_{0+}^\alpha (g-\lambda y),
  \]
  where
  \[
    c = \left( \I_{0+}^{1-\alpha} (y-y_0) \right)(0).
  \]
  Since $ y \in H^{\alpha/2}(0,T) $ and \cref{lem:regu} implies $
  \I_{0+}^\alpha(g-\lambda y) \in H^\alpha(0,T) $, it is evident that $ c= 0 $,
  and hence
  \[
    y = y_0 + \I_{0+}^\alpha (g-\lambda y) \in H^\alpha(0,T).
  \]
  Furthermore, \cref{lem:regu} indicates that
  \begin{equation}
    \label{eq:xy-1}
    \nm{y}_{H^\alpha(0,T)} \leqslant
    C_{\alpha,T} \left(
      \snm{y_0} + \nm{g-\lambda y}_{L^2(0,T)}
    \right).
  \end{equation}

  Now, we proceed to prove \cref{eq:regu_ode_II}, and since the techniques used
  below are similar to that used in the proof of \cref{lem:regu_ode_I}, the
  forthcoming proof will be brief. Firstly, let us prove that
  \begin{equation}
    \label{eq:regu_ode_II-1}
    \snm{y}_{H^{\alpha/2}(0,T)}^2 + \lambda \nm{y}_{L^2(0,T)}^2
    \leqslant C_{\alpha,\theta,T}
    \left(
      \lambda^{-1} \nm{g}_{L^2(0,T)}^2 + \lambda^{-\theta} \snm{y_0}^2
    \right).
  \end{equation}
  Using the standard estimate that (\cite[Lemma 16.3]{Tartar2007})
  \[
    \int_0^T t^{-(1-\theta)\alpha} \snm{y(t)}^2 \, \mathrm{d}t
    \leqslant C_{\alpha,\theta} \snm{y}_{H^{(1-\theta)\alpha/2}(0,T)}^2,
  \]
  by the Cauchy-Schwarz inequality we obtain
  \[
    \dual{\D_{0+}^\alpha y_0, y}_{H^{\alpha/2}(0,T)} =
    \frac{y_0}{\Gamma(1-\alpha)} \dual{t^{-\alpha}, y}_{(0,T)}
    \leqslant C_{\alpha,\theta,T} \snm{y_0}
    \snm{y}_{H^{(1-\theta)\alpha/2}(0,T)}.
  \]
  Since
  \[
    \snm{y}_{H^{(1-\theta)\alpha/2}(0,T)} \leqslant
    \nm{y}_{L^2(0,T)}^\theta
    \snm{y}_{H^{\alpha/2}(0,T)}^{1-\theta},
  \]
  it follows that
  \begin{equation}
    \label{eq:703}
    \begin{aligned}
      \dual{\D_{0+}^\alpha y_0, y}_{H^{\alpha/2}(0,T)} &
      \leqslant C_{\alpha,\theta,T}  \snm{y_0}
      \nm{y}_{L^2(0,T)}^\theta
      \snm{y}_{H^{\alpha/2}(0,T)}^{1-\theta} \\
      & \leqslant C_{\alpha,\theta,T} \snm{y_0}
      \lambda^{-\theta/2} \left( \lambda^{1/2}\nm{y}_{L^2(0,T)} \right)^\theta
      \snm{y}_{H^{\alpha/2}(0,T)}^{1-\theta} \\
      & \leqslant
      C_{\alpha,\theta,T} \snm{y_0} \lambda^{-\theta/2}
      \left(
        \snm{y}_{H^{\alpha/2}(0,T)} + \lambda^{1/2} \nm{y}_{L^2(0,T)}
      \right).
    \end{aligned}
  \end{equation}
   In addition, inserting $ z=y $ into
  \cref{eq:frac_ode_weak} yields
  \begin{align*}
    & \snm{y}_{H^{\alpha/2}(0,T)}^2 + \lambda \nm{y}_{L^2(0,T)}^2 \\
    \leqslant{} &
    C_\alpha \left(
      \lambda^{-1} \nm{g}_{L^2(0,T)}^2 +
      \dual{\D_{0+}^\alpha y_0, y}_{H^{\alpha/2}(0,T)}
    \right).
  \end{align*}
  Consequently, inserting \cref{eq:703} into the above inequality and applying
  the Young's inequality with $ \epsilon $, we obtain \cref{eq:regu_ode_II-1}.

  Secondly, let us prove that
  \begin{equation}
    \label{eq:regu_ode_II-2}
    \nm{y}_{H^\alpha(0,T)}^2 \leqslant C_{\alpha,\theta,T}
    \left(
      \nm{g}_{L^2(0,T)}^2 + \lambda^{1-\theta} \snm{y_0}^2
    \right).
  \end{equation}
  Multiplying both sides of \cref{eq:frac_ode} by $ \D_{0+}^\alpha(y-y_0) $ and
  integrating over $ (0,T) $, we obtain
  \[
\begin{aligned}
{}&    \nm{\D_{0+}^\alpha (y-y_0)}_{L^2(0,T)}^2 +
\lambda \snm{y}_{H^{\alpha/2}(0,T)}^2\\ 
\leqslant{}&
C_{\alpha,T} \left(
\nm{g}_{L^2(0,T)}^2 +
\lambda \dual{\D_{0+}^\alpha y_0,y}_{H^{\alpha/2}(0,T)}
%\lambda \snm{y}_{H^{\alpha/2}(0,T)} \snm{y_0}
\right),
\end{aligned}
  \]
  so that from \cref{eq:703,eq:regu_ode_II-1} it follows that
  \begin{align*}
    & \nm{\D_{0+}^\alpha (y-y_0)}_{L^2(0,T)}^2 +
    \lambda \snm{y}_{H^{\alpha/2}(0,T)}^2 \\
    \leqslant{} &
    C_{\alpha,\theta,T} \left(
      \nm{g}_{L^2(0,T)}^2 + \snm{y_0} \lambda^{1-\theta/2}
      \left(
        \lambda^{-1/2} \nm{g}_{L^2(0,T)} + \lambda^{-\theta/2} \snm{y_0}
      \right)
    \right) \\
    \leqslant{} &
    C_{\alpha,\theta,T} \left(
      \nm{g}_{L^2(0,T)}^2 + \snm{y_0} \lambda^{1/2-\theta/2}
        \nm{g}_{L^2(0,T)} + \lambda^{1-\theta} \snm{y_0}^2
    \right) \\
    \leqslant{} &
    C_{\alpha,\theta,T} \left(
      \nm{g}_{L^2(0,T)}^2 + \lambda^{1-\theta} \snm{y_0}^2
    \right).
  \end{align*}
  Therefore, combing \cref{eq:frac_ode,eq:xy-1} yields \cref{eq:regu_ode_II-2}.

  Finally, using the same technique as that used to derive
  \cref{eq:regu_ode_I-34}, by \cref{eq:regu_ode_II-1,eq:regu_ode_II-2} we
  conclude that
  \begin{align*}
    & \nm{y}_{H^{\alpha(1+\theta)/2}(0,T)}^2 +
    \lambda \nm{y}_{H^{\alpha\theta/2}(0,T)}^2 \\
    \leqslant{} &
    C_{\alpha,\theta,T} \left(
      \lambda^{\theta-1} \nm{g}_{L^2(0,T)}^2 + \snm{y_0}^2
    \right),
  \end{align*}
  which, together with \cref{eq:regu_ode_II-1,eq:regu_ode_II-2}, yields
  inequality \cref{eq:regu_ode_II}. This theorem is thus proved.
\end{proof}

\begin{lem}
  \label{lem:regu_ode_III}
  Assume that $ 1/2 < \alpha < 1 $. If $ y_0 = 0 $ and $ g \in H^{1-\alpha}(0,T)
  $, then
  \begin{equation}
    \label{eq:regu_ode_III}
    \nm{y}_{H^1(0,T)} + \lambda^{1/2} \nm{y}_{H^{1-\alpha/2}(0,T)} +
    \lambda \nm{y}_{L^2(0,T)} \leqslant
    C_{\alpha,T} \nm{g}_{H^{1-\alpha}(0,T)}.
  \end{equation}
\end{lem}
\begin{proof}
  Let us first prove that
  \begin{equation}
    \label{eq:regu_ode_III-1}
    y' = \D_{0+}^{1-\alpha} (g - \lambda y).
  \end{equation}
  Since we have already proved
  \[
    y = \I_{0+}^\alpha (g-\lambda y)
  \]
  in the proof of \cref{lem:regu_ode_II}, by \cref{lem:regu} we obtain $ y \in
  H^1(0,T) $. Moreover, because
  \[
    \snm{ \I_{0+}^\alpha(g-\lambda y)(s) }
    \leqslant \frac{s^{\alpha-1/2}}{\Gamma(\alpha) \sqrt{2\alpha-1}}
    \nm{g-\lambda y}_{L^2(0,s)}, \quad 0 < s < T,
  \]
  we have
  \[
    \lim_{s \to 0+}\I_{0+}^\alpha(g-\lambda y)(s) = 0.
  \]
  Consequently, we obtain $ y(0) = 0 $ and hence
  \[
    \D_{0+}^\alpha y = \D_{0+}^\alpha \I_{0+} y' = \I_{0+}^{1-\alpha} y',
  \]
  which, together with \cref{eq:frac_ode}, yields
  \[
    \I_{0+}^{1-\alpha} y' = g - \lambda y.
  \]
  Therefore,
  \[
    y' = \D \I_{0+} y' = \D \I_{0+}^\alpha \I_{0+}^{1-\alpha} y' =
    \D_{0+}^{1-\alpha} \I_{0+}^{1-\alpha} y' =
    \D_{0+}^{1-\alpha} (g-\lambda y).
  \]
  This proves equality \cref{eq:regu_ode_III-1}.

  Then, let us prove \cref{eq:regu_ode_III}. Multiplying both sides of
  \cref{eq:regu_ode_III-1} by $ y' $ and integrating over $ (0,T) $ yields
  \[
    \nm{y'}_{L^2(0,T)}^2 + \lambda \dual{\D_{0+}^{1-\alpha} y, y'}_{(0,T)} =
    \dual{\D_{0+}^{1-\alpha} g, y'}_{(0,T)},
  \]
  so that
  \[
    \nm{y'}_{L^2(0,T)}^2 + \lambda \dual{\D_{0+}^{1-\alpha} y, y'}_{(0,T)}
    \leqslant C_{\alpha,T} \nm{g}_{H^{1-\alpha}(0,T)}^2,
  \]
  by the Cauchy-Schwarz inequality, \cref{lem:coer} and the Young's inequality
  with $ \epsilon $. Additionally, using the fact that $ y \in H^1(0,T) $ with
  $ y(0) = 0 $ gives
  \[
    \D_{0+}^{1-\alpha} y = \D \I_{0+}^\alpha y =
    \I_{0+}^\alpha y',
  \]
  so that
  \[
    \dual{\D_{0+}^{1-\alpha}y, y'}_{(0,T)}
    \geqslant C_{\alpha,T} \nm{y}_{H^{1-\alpha/2}(0,T)}^2,
  \]
  by \cref{lem:key,lem:xy}. Therefore,
  \[
    \nm{y'}_{L^2(0,T)}^2 + \lambda \nm{y}_{H^{1-\alpha/2}(0,T)}^2
    \leqslant C_{\alpha,T} \nm{g}_{H^{1-\alpha}(0,T)}^2,
  \]
  and hence, as \cref{lem:regu_ode_II} implies
  \[
    \lambda \nm{y}_{L^2(0,T)} \leqslant C_{\alpha,T} \nm{g}_{L^2(0,T)},
  \]
  we readily obtain \cref{eq:regu_ode_III}. This completes the proof.
\end{proof}
It is clear that we can represent $ u $ in the following form
\[
u(t) = \sum_{i=0}^\infty y_i(t) \phi_i, \quad 0 < t < T,
\]
where $ y_i $ solves problem \cref{eq:frac_ode_weak} with $ \lambda $, $g$ and $y_0$ replaced by $
\lambda_i $, $f_i$ and $u_{0,i}$, respectively. Here, note that $f_i$ and $ u_{0,i} $ are the coordinates of $f$ and $u_0$ respectively under the orthonormal basis $\{\phi_i:i\in\mathbb N\}$. Therefore, by the above three lemmas we readily conclude the following
regularity estimates for problem \cref{eq:weak_form}.
%Then, from the above three lemmas, we readily conclude the following regularity
%estimates for problem \cref{eq:weak_form}, using the standard Galerkin method.
\begin{thm}
  \label{thm:regu_pde_I}
  Assume that $ 0 < \alpha < 1/2 $. If $ f
  \in L^2(0,T;H^{-\beta}(\Omega)) $ and $ u_0 \in H^{-\beta}(\Omega) $ with $ 0 \leqslant \beta < 1 $, then
  \begin{equation*}
    \begin{aligned}
      & \snm{u}_{H^{\alpha(1-\beta/2)}(0,t;L^2(\Omega))} +
      \snm{y}_{H^{\alpha/2}(0,t; \dot H^{1-\beta}(\Omega))} +
      \snm{u}_{H^{\alpha(1-\beta)/2}(0,t;\dot H^1(\Omega))} \\
      & {} + \nm{u}_{L^2(0,t;\dot H^{2-\beta}(\Omega))}
      \leqslant C_{\alpha,\Omega} \left(
        \nm{f}_{L^2(0,t;H^{-\beta}(\Omega))} +
        t^{1/2-\alpha} \nm{u_0}_{H^{-\beta}(\Omega)}
      \right)
    \end{aligned}
  \end{equation*}
  for all $ 0 < t < T $.
\end{thm}

\begin{thm}
  \label{thm:regu_pde_II}
  Assume that $ 1/2 \leqslant \alpha < 1 $. If $ f
  \in L^2(0,T;H^{-\beta}(\Omega)) $ with $ 2 - 1/\alpha < \beta < 1 $ and $ u_0 \in L^2(\Omega) $, then
  \begin{equation*}
%    \label{eq:regu_pde_II}
    \begin{aligned}
      & \nm{u}_{H^{\alpha(1-\beta/2)}(0,T;L^2(\Omega))} +
      \snm{u}_{H^{\alpha/2}(0,T; \dot H^{1-\beta}(\Omega))} +
      \nm{u}_{H^{\alpha(1-\beta)/2}(0,T; \dot H^1(\Omega))} \\
      & {} +
      \nm{u}_{L^2(0,T; \dot H^{2-\beta}(\Omega))} \leqslant
      C_{\alpha,\beta,T,\Omega} \left(
        \nm{f}_{L^2(0,T;H^{-\beta}(\Omega))} + \nm{u_0}_{L^2(\Omega)}
      \right).
    \end{aligned}
  \end{equation*}
  Moreover, if $ u_0 = 0 $ and $ f \in L^2(0,T; H^{-\beta}(\Omega)) $ with $ 0
  \leqslant \beta < 1 $, then the above estimate also holds.
\end{thm}

\begin{thm}
  Assume that $ 1/2 < \alpha < 1 $. If $ u_0 = 0 $ and $ f \in H^{1-\alpha}(0,T;
  L^2(\Omega)) $, then
  \begin{equation*}
    \begin{aligned}
      & \nm{u}_{H^1(0,T; L^2(\Omega))} +
      \nm{u}_{H^{1-\alpha/2}(0,T; \dot H^1(\Omega))} +
      \nm{u}_{L^2(0,T; \dot H^2(\Omega))} \\
      \leqslant{} &
      C_{\alpha,T,\Omega}
      \nm{f}_{H^{1-\alpha}(0,T; L^2(\Omega))}.
    \end{aligned}
  \end{equation*}
\end{thm}

%\begin{rem}
  %Using the techniques used in the proofs of
  %\cref{lem:regu_ode_I,lem:regu_ode_II,lem:regu_ode_III}, we can easily
  %establish the regularity estimates for problem \cref{eq:weak_form} under the
  %condition that $ u_0 $ and $ f $ possess other regularity.
%\end{rem}

\section{Convergence}
\label{sec:conv}
We assume that $ u $ and $ U $ are respectively the solutions to problems
\cref{eq:weak_form} and \cref{eq:algor}, and by $ a \lesssim b $ we mean that there
exists a generic positive constant $ C $, independent of $ h $, $ \tau $ and $ u $,
such that $ a \leqslant C b $. The main task of this section is to prove the
following a priori error estimates.
\begin{thm}
  \label{thm:conv_I}
  Assume that $ 0 < \alpha < 1/2 $ and $ 0 \leqslant \beta < 1 $. If $ u_0 \in
  H^{-\beta}(\Omega) $ and $ f \in L^2(0,T; H^{-\beta}(\Omega)) $, then
  %\begin{align}
    %\frac{ \nm{u-U}_W }
    %{
      %\nm{f}_{L^2(0,T;H^{-\beta}(\Omega))} +
      %\nm{u_0}_{H^{-\beta}(\Omega)}
    %} & \lesssim
    %h^{1-\beta} + \tau^{\alpha(1-\beta)/2},
    %\label{eq:conv_I_H1} \\
    %\frac{ \nm{u-U}_{L^2( 0,T;L^2(\Omega) )} }
    %{
      %\nm{f}_{L^2(0,T;H^{-\beta}(\Omega))} +
      %\nm{u_0}_{H^{-\beta}(\Omega)}
    %} & \lesssim
    %h^{2-\beta} + \tau^{\alpha(1-\beta/2)}.
    %\label{eq:conv_I_L2}
  %\end{align}
  \begin{align}
    & \nm{u-U}_{L^2(0,T;\dot H^1(\Omega))} \notag \\
    \lesssim{} &
    \left( h^{1-\beta} + \tau^{\alpha(1-\beta)/2} \right)
    \left(
      \nm{f}_{L^2(0,T;H^{-\beta}(\Omega))} +
      \nm{u_0}_{H^{-\beta}(\Omega)}
    \right),
    \label{eq:conv_I_H1} \\
    & \nm{u-U}_{L^2( 0,T;L^2(\Omega) )} \notag \\
    \lesssim{} &
    \left( h^{2-\beta} + \tau^{\alpha(1-\beta/2)} \right)
    \left(
      \nm{f}_{L^2(0,T;H^{-\beta}(\Omega))} +
      \nm{u_0}_{H^{-\beta}(\Omega)}
    \right).
    \label{eq:conv_I_L2}
  \end{align}
\end{thm}

\begin{thm}
  \label{thm:conv_II}
  Assume that $ 1/2 \leqslant \alpha < 1 $ and $ 2-1/\alpha<\beta\leqslant 1
  $. If $ u_0 \in L^2(\Omega) $ and $ f \in L^2(0,T; H^{-\beta}(\Omega)) $,
  then
  \begin{align*}
    & \nm{u-U}_{L^2(0,T;\dot H^1(\Omega))} \notag \\
    \lesssim{} &
    \left( h^{1-\beta} + \tau^{\alpha(1-\beta)/2} \right)
    \left(
      \nm{f}_{L^2(0,T;H^{-\beta}(\Omega))} +
      \nm{u_0}_{L^2(\Omega)}
    \right), \\
    & \nm{u-U}_{L^2(0,T;L^2(\Omega))} \notag \\
    \lesssim{} &
    \left( h^{2-\beta} + \tau^{\alpha(1-\beta/2)} \right)
    \left(
      \nm{f}_{L^2(0,T;H^{-\beta}(\Omega))} +
      \nm{u_0}_{L^2(\Omega)}
    \right).
  \end{align*}
  Moreover, if $ u_0 = 0 $ and $ f \in L^2(0,T;H^{-\beta}(\Omega)) $, then
  the above two estimates also hold for all $ 0 \leqslant \beta < 1 $.
\end{thm}

\begin{thm}
  \label{thm:conv_III}
  Assume that $ 1/2 < \alpha < 1 $. If $ u_0 = 0 $ and $ f \in
  H^{1-\alpha}(0,T;L^2(\Omega)) $, then
  \begin{align*}
      \nm{u-U}_{L^2(0,T;L^2(\Omega))} & \lesssim
  \left( h^2 + \tau \right)
  \nm{f}_{H^{1-\alpha}(0,T;L^2(\Omega))},\\
    \nm{u-U}_{L^2(0,T;\dot H^1(\Omega))} & \lesssim
    \left( h + \tau^{1-\alpha/2} \right)
    \nm{f}_{H^{1-\alpha}(0,T;L^2(\Omega))}.
  \end{align*}
\end{thm}

Since the proofs of \cref{thm:conv_II,thm:conv_III} are similar to that of
\cref{thm:conv_I}, below we only show the latter. To this end, we start by
introducing two interpolation operators. For any $ v \in L^1(0,T;X) $ with $ X $
being a separable Hilbert space, define $ P_\tau v $ by
\[
  \left( P_\tau v \right)|_{I_j} :=
  \frac1{\tau_j} \int_{I_j} v(t) \, \mathrm{d}t,
  \quad 1 \leqslant j \leqslant J.
\]
Let $ P_h: L^2(\Omega) \to \mathcal S_h $ be the well-known Cl\'ement interpolation
operator. For the above two operators, we have the following standard estimates
\cite{Cl1975Approximation,Ciarlet2002}: if $ 0 \leqslant \beta \leqslant 1 $ and $ \beta \leqslant \gamma
\leqslant 2 $, then
\[
  \nm{(I-P_h)v}_{H^\beta(\Omega)} \lesssim
  h^{\gamma-\beta} \nm{v}_{\dot H^\gamma(\Omega)},
  \quad \forall v \in \dot H^\gamma(\Omega);
\]
if $ 0 \leqslant \beta < 1/2 $ and $ \beta \leqslant \gamma \leqslant 1 $, then
\[
  \nm{(I-P_\tau)w}_{H^\beta(0,T)} \lesssim
  \tau^{\gamma-\beta} \nm{w}_{H^\gamma(0,T)},
  \quad \forall w \in H^\gamma(0,T).
\]
For clarity, below we shall use the above two estimates implicitly.

\medskip\noindent{\bf Proof of \cref{thm:conv_I}.} Let us first prove
  \cref{eq:conv_I_H1}. By \cref{lem:coer}, a standard procedure yields that
  \[
    \nm{u-U}_W \lesssim \nm{u-P_\tau P_h u}_W,
  \]
 then using the triangle inequality gives
  \begin{align*}
    \nm{u-U}_{W}
    & \lesssim
    \snm{(I-P_h)u}_{H^{\alpha/2}(0,T;L^2(\Omega))} +
    \snm{(I-P_\tau)P_hu}_{H^{\alpha/2}(0,T;L^2(\Omega))} \\
    & \quad + {} \nm{(I-P_h)u}_{L^2(0,T;\dot H^1(\Omega))} +
    \nm{(I-P_\tau)P_hu}_{L^2(0,T;\dot H^1(\Omega))}.
  \end{align*}
Since
  \begin{align*}
  \snm{(I-P_\tau)P_hu}_{H^{\alpha/2}(0,T;L^2(\Omega))}
  &\leqslant \snm{(I-P_\tau)u}_{H^{\alpha/2}(0,T;L^2(\Omega))},\\
\nm{(I-P_\tau)P_hu}_{L^2(0,T;\dot H^1(\Omega))} &\lesssim
\nm{(I-P_\tau)u}_{L^2(0,T; \dot H^1(\Omega))},
\end{align*}
it follows that
 \begin{equation}\label{eq:u-U_W}
  \begin{aligned}
\nm{u-U}_{W}& \lesssim
\snm{(I-P_h)u}_{H^{\alpha/2}(0,T;L^2(\Omega))} +
\snm{(I-P_\tau)u}_{H^{\alpha/2}(0,T;L^2(\Omega))} \\
& \quad + {} \nm{(I-P_h)u}_{L^2(0,T;\dot H^1(\Omega))} +
\nm{(I-P_\tau) u}_{L^2(0,T; \dot H^1(\Omega))}.
\end{aligned}
  \end{equation}
  Therefore, \cref{eq:conv_I_H1} is a direct consequence of
  \cref{thm:regu_pde_I} and the following estimates:
  \begin{align*}
      \nm{(I-P_h)u}_{L^2(0,T;\dot H^1(\Omega))} &
  \lesssim h^{1-\beta} \nm{u}_{ L^2( 0,T; \dot H^{2-\beta}(\Omega) ) }, \\
    \snm{(I-P_h)u}_{ H^{\alpha/2}( 0,T;L^2(\Omega) ) } &
    \lesssim h^{1-\beta} \snm{u}_{ H^{\alpha/2}( 0,T; \dot H^{1-\beta}(\Omega) ) }, \\
        \snm{(I-P_\tau)u}_{ H^{\alpha/2}( 0,T;L^2(\Omega) ) } &
    \lesssim \tau^{\alpha(1-\beta)/2}
    \snm{u}_{ H^{\alpha(1-\beta/2)}( 0,T;L^2(\Omega) ) },\\
    \nm{(I-P_\tau)u}_{ L^2( 0,T; \dot H^1(\Omega) ) } &
    \lesssim \tau^{\alpha(1-\beta)/2} \snm{u}_{ H^{\alpha(1-\beta)/2}( 0,T; \dot H^1(\Omega) ) }.
  \end{align*}

  Then let us prove \cref{eq:conv_I_L2}. By \cref{lem:coer}, the well known
  Lax-Milgram theorem implies that there exists a unique $ z \in W $ such that
  \[
    \dual{\D_{T-}^\alpha z, v}_{H^{\alpha/2}(0,T;L^2(\Omega))} +
    \dual{\nabla z, \nabla v}_{\Omega \times (0,T)}
    = \dual{u-U,v}_{\Omega \times (0,T)}
  \]
  for all $ v \in W $. Substituting $ v = u-U $ into the above equation yields
  \begin{align*}
    \nm{u-U}_{L^2(0,T;L^2(\Omega))}^2
    &= \dual{\D_{T-}^\alpha z, u-U}_{H^{\alpha/2}(0,T;L^2(\Omega))} +
    \dual{\nabla z,\nabla(u-U)}_{\Omega \times (0,T)} \\
    &=
    \dual{\D_{0+}^\alpha(u-U), z}_{H^{\alpha/2}(0,T;L^2(\Omega))} +
    \dual{\nabla(u-U),\nabla z}_{\Omega \times (0,T)},
  \end{align*}
  by \cref{lem:coer}. Setting $ Z = P_\tau P_h z $, as combining
  \cref{eq:weak_form,eq:algor} gives
  \[
    \dual{\D_{0+}^\alpha(u-U), Z}_{H^{\alpha/2}(0,T;L^2(\Omega))} +
    \dual{\nabla(u-U), \nabla Z}_{\Omega \times (0,T)}
    =0,
  \]
  we obtain
  \begin{align*}
    & \nm{u-U}_{L^2(0,T;L^2(\Omega))}^2 \\
    ={} &
    \dual{\D_{0+}^\alpha(u-U), z-Z}_{H^{\alpha/2}(0,T;L^2(\Omega))} +
    \dual{\nabla(u-U),\nabla (z-Z)}_{\Omega \times (0,T)}.
  \end{align*}
  Then \cref{lem:coer} implies that
  \begin{equation}
    \label{eq:conv_I_L2-1}
    \begin{aligned}
      \nm{u-U}_{L^2(0,T;L^2(\Omega))}^2
      & \leqslant
      \snm{u-U}_{H^{\alpha/2}(0,T;L^2(\Omega))}
      \snm{z-Z}_{H^{\alpha/2}(0,T;L^2(\Omega))} \\
      & \quad {} +
      \nm{u-U}_{L^2(0,T;\dot H^1(\Omega))}
      \nm{z-Z}_{L^2(0,T;\dot H^1(\Omega))}\\
      & \leqslant \nm{u-U}_{W} \nm{z-Z}_{W}.
    \end{aligned}
  \end{equation}
  Similarly to the regularity estimate in \cref{thm:regu_pde_I}, we have
  \[
    \nm{z}_{H^\alpha(0,T;L^2(\Omega))} +
    \snm{z}_{H^{\alpha/2}(0,T;\dot H^1(\Omega))} +
    \nm{z}_{L^2(0,T; \dot H^2(\Omega))} \lesssim
    \nm{u-U}_{L^2(0,T;L^2(\Omega))},
  \]
  so that proceeding as in the proof of \cref{eq:conv_I_H1} yields
  \[
    \nm{z-Z}_{W} \lesssim (h+\tau^{\alpha/2})
    \nm{u-U}_{L^2(0,T;L^2(\Omega))}.
  \]
  Collecting the above estimate, \cref{eq:u-U_W,eq:conv_I_L2-1} gives
    \begin{align*}
  & \nm{u-U}_{L^2( 0,T;L^2(\Omega) )} \notag \\
  \lesssim{} &
   (h+\tau^{\alpha/2})\left( h^{1-\beta} + \tau^{\alpha(1-\beta)/2} \right)
  \left(
  \nm{f}_{L^2(0,T;H^{-\beta}(\Omega))} +
  \nm{u_0}_{H^{-\beta}(\Omega)}
  \right).
%    \lesssim{} &
%\left( h^{2-\beta} + \tau^{\alpha(1-\beta/2)}+ h\tau^{\alpha(1-\beta)/2}+  h^{1-\beta} \tau^{\alpha/2} \right)
%  \left(
%  \nm{f}_{L^2(0,T;H^{-\beta}(\Omega))} +
%  \nm{u_0}_{H^{-\beta}(\Omega)}
%  \right).
%  \label{eq:conv_I_L2}
  \end{align*}
Therefore, \cref{eq:conv_I_L2} is a direct
  consequence of the following two estimates:
  \begin{align*}
    h\tau^{\alpha(1-\beta)/2} & =
    \left( h^{2-\beta} \right)^{1/(2-\beta)}
    \left( \tau^{\alpha(1-\beta/2)} \right)^{1-1/(2-\beta)} \\
    & \leqslant h^{2-\beta}/(2-\beta) + (1-1/(2-\beta))
    \tau^{\alpha(1-\beta/2)}, \\
    h^{1-\beta} \tau^{\alpha/2} &=
    \left( h^{2-\beta} \right)^{(1-\beta)/(2-\beta)}
    \left( \tau^{\alpha(1-\beta/2)} \right)^{1-(1-\beta)/(2-\beta)} \\
    & \leqslant
    (1-\beta)/(2-\beta) h^{2-\beta} +
    (1-(1-\beta)/(2-\beta)) \tau^{\alpha(1-\beta/2)}.
  \end{align*}
  This completes the proof.
\hfill\ensuremath{\blacksquare}

\section{Numerical Results}
\label{sec:numer}
This section performs some numerical experiments to verify our theoretical results in
one-dimensional space. We set $ \Omega = (0,1) $, $ T = 1 $ and
\begin{align*}
	\mathcal E_1 &: = \nm{\widetilde u - U}_{L^2(0,T;H_0^1(\Omega))}, \\
	\mathcal E_2 &: = \nm{\widetilde u - U}_{L^2(0,T;L^2(\Omega))},
\end{align*}
where $ \widetilde u $ is a reference solution.
\vskip 0.2cm
\medskip\noindent{\bf Experiment 1.} This experiment verifies \cref{thm:conv_I} under the
condition that
\begin{alignat*}{2}
	u_0(x) & := x^r, & \quad & 0 < x < 1, \\
	f(x,t) & := x^r t^{-0.49}, & \quad & 0 < x < 1,\ 0 < t < T.
\end{alignat*}
We first summarize the numerical results in \cref{tab:ex1-space} as follows.
\begin{itemize}
	\item If $ r=-0.8$, then
	\[
	u_0 \in H^{-\beta}(\Omega) \quad\text{ and }
	\quad f \in L^2(0,T;H^{-\beta}(\Omega))
	\]
	for all $ \beta > 0.3 $. Therefore, \cref{thm:conv_I} indicates that the spatial
	convergence orders of $ \mathcal E_1 $ and $ \mathcal E_2 $ are close to $ \mathcal
	O(h^{0.7}) $ and $ \mathcal O(h^{1.7}) $, respectively. This is confirmed by the
	numerical results.
	\item If $ r=-0.99 $, then
	\[
	u_0 \in H^{-\beta}(\Omega) \quad\text{ and }
	\quad f \in L^2(0,T;H^{-\beta}(\Omega))
	\]
	for all $ \beta > 0.49 $. Therefore, \cref{thm:conv_I} indicates that the spatial
	convergence orders of $ \mathcal E_1 $ and $ \mathcal E_2 $ are close to $ \mathcal
	O(h^{0.51}) $ and $ \mathcal O(h^{1.51}) $, respectively. This agrees well with the
	numerical results.
\end{itemize}

In the case of $ \alpha = 0.4 $ and $ r=-0.49 $, \cref{thm:conv_I} indicates that the
temporal convergence orders of $ \mathcal E_1 $ and $ \mathcal E_2 $ are close to $
\mathcal O(\tau^{0.2}) $ and $ \mathcal O(\tau^{0.4}) $, respectively. In the case of $
\alpha = 0.4 $ and $ r=0.99 $, \cref{thm:conv_I} indicates that the temporal convergence
orders of $ \mathcal E_1 $ and $ \mathcal E_2 $ are close to $ \mathcal O(\tau^{0.1}) $
and $ \mathcal O(\tau^{0.3}) $, respectively. These theoretical results coincide with the numerical results in
\cref{tab:ex1-time}.
\begin{table}[ht]
	\caption{ Convergence history with
		$ \tau = 2^{-15}$ ( $ \widetilde u $ is the numerical solution at $
		h=2^{-11} $).
	}
	\label{tab:ex1-space} \small\setlength{\tabcolsep}{1.5pt}
	\begin{tabular}{lcccccccccccc}
		\toprule
		&& & & \multicolumn{4}{c}{$\alpha=0.2$} & &
		\multicolumn{4}{c}{$\alpha=0.4$} \\
		\cmidrule{5-8} \cmidrule{10-13}
		&& $h$ &\phantom{a} & $ \mathcal E_1 $ & Order & $ \mathcal E_2 $ & Order & \phantom{aa}
		& $ \mathcal E_1 $ & Order & $ \mathcal E_2 $ & Order \\
		\midrule
		\multirow{5}{*}{$r\!=\!-0.8$}
		&  & $2^{-3}$ &  & 7.56e-1 & --   & 1.15e-2 & --   &  & 8.12e-1 & --   & 2.87e-2 & --   \\
		&  & $2^{-4}$ &  & 4.78e-1 & 0.66 & 3.64e-3 & 1.66 &  & 5.23e-1 & 0.64 & 9.42e-3 & 1.61 \\
		&  & $2^{-5}$ &  & 2.99e-1 & 0.68 & 1.14e-3 & 1.68 &  & 3.30e-1 & 0.66 & 3.02e-3 & 1.64 \\
		&  & $2^{-6}$ &  & 1.85e-1 & 0.69 & 3.53e-4 & 1.69 &  & 2.06e-1 & 0.68 & 9.51e-4 & 1.67 \\
		\midrule
		\multirow{5}{*}{$r\!=\!-0.99$}
		&  & $2^{-3}$ &  & 1.51e-0 & --   & 5.10e-2 & --   &  & 1.64e-0 & --   & 5.45e-2 & --   \\
		&  & $2^{-4}$ &  & 1.07e-0 & 0.49 & 1.84e-2 & 1.47 &  & 1.19e-0 & 0.47 & 2.01e-2 & 1.44 \\
		&  & $2^{-5}$ &  & 7.54e-1 & 0.41 & 6.53e-3 & 1.49 &  & 8.42e-1 & 0.49 & 7.25e-3 & 1.47 \\
		&  & $2^{-6}$ &  & 5.27e-1 & 0.52 & 2.31e-3 & 1.50 &  & 5.91e-1 & 0.51 & 2.58e-3 & 1.49 \\
		\bottomrule
	\end{tabular}
\end{table}

\begin{table}[ht]
	\caption{Convergence history with 
		$\alpha=0.4$ and $ h=2^{-10}$ ($ \widetilde u $ is the numerical solution at $ \tau=2^{-17} $).
	}
	\label{tab:ex1-time}
	\small\setlength{\tabcolsep}{1.5pt}
	\begin{tabular}{ccccccccccc}
		\toprule
		\multicolumn{5}{c}{$r=-0.49$} & & \multicolumn{5}{c}{$r=-0.99$} \\
		\cmidrule{1-5} \cmidrule{7-11}
		$\tau $ & $ \mathcal E_1 $ & Order & $ \mathcal E_2 $ & Order & \phantom{aa} &
		$\tau $ & $ \mathcal E_1 $ & Order & $ \mathcal E_2 $ & Order \\
		$2^{-5}$ & 4.54e-1 & --   & 1.20e-2 & --   &  & $2^{-3}$ & 1.80 & --   & 3.49e-1 & --   \\
		$2^{-6}$ & 3.77e-1 & 0.27 & 9.53e-2 & 0.33 &  & $2^{-4}$ & 1.62 & 0.15 & 2.93e-1 & 0.25 \\
		$2^{-7}$ & 3.11e-1 & 0.28 & 7.39e-2 & 0.37 &  & $2^{-5}$ & 1.45 & 0.16 & 2.42e-1 & 0.28 \\
		$2^{-8}$ & 2.56e-1 & 0.28 & 5.63e-2 & 0.39 &  & $2^{-6}$ & 1.30 & 0.16 & 1.96e-1 & 0.30 \\
		\bottomrule
	\end{tabular}
\end{table}

\vskip 0.1cm
\noindent{\bf Experiment 2.} This experiment verifies \cref{thm:conv_II} under the
condition that
\begin{alignat*}{2}
	u_0(x) & := cx^{-0.49}, & \quad & 0 < x < 1, \\
	f(x,t) & := x^{-0.8} t^{-0.49}, & \quad & 0 < x < 1,\ 0 < t < T.
\end{alignat*}
For $ \alpha = 0.7 $, \cref{thm:conv_II} implies the following results: if $ c=0 $, then
\[
\mathcal E_1 \approx\mathcal O(h^{0.7})\quad\mathrm{and}\quad \mathcal E_2  \approx\mathcal O(h^{1.7});
\]
if $ c=1 $, then
\[
\mathcal E_1 \approx\mathcal O(h^{0.43})\quad\mathrm{and}\quad \mathcal E_2  \approx\mathcal O(h^{1.43}).
\] 
These theoretical results are confirmed by
the numerical results in \cref{tab:ex2-space}.

For $ \alpha=0.8 $, \cref{thm:conv_II} implies the following results: if $ c = 0 $, then
the temporal convergence orders of $ \mathcal E_1 $ and $ \mathcal E_2 $ are close to $
\mathcal O(\tau^{0.28}) $ and $ \mathcal O(\tau^{0.68}) $, respectively; if $ c = 1 $,
then the temporal convergence orders of $ \mathcal E_1 $ and $ \mathcal E_2 $ are close to
$ \mathcal O(\tau^{0.1}) $ and $ \mathcal O(\tau^{0.5}) $, respectively. These  theoretical results  are verified
by \cref{tab:ex2-time}.
\begin{table}[H]
	\caption{ Convergence history with
		$\alpha=0.7$ and $ \tau = 2^{-15}$ ($ \widetilde u $ is the numerical solution at  $ h=2^{-11} $).
	}
	\label{tab:ex2-space}
	\small\setlength{\tabcolsep}{1.5pt}
	\begin{tabular}{ccccccccccc}
		\toprule
		& & \multicolumn{4}{c}{$c=0$} & &
		\multicolumn{4}{c}{$c=1$} \\
		\cmidrule{3-6} \cmidrule{8-11}
		$h$ & & $ \mathcal E_1 $ & Order & $ \mathcal E_2 $ & Order & \phantom{aa}
		& $ \mathcal E_1 $ & Order & $ \mathcal E_2 $ & Order \\
		\midrule
		$2^{-2}$ &  & 7.50e-1 & --   & 5.07e-2 & --   &  & 1.76e-0 & --   & 1.04e-1 & --   \\
		$2^{-3}$ &  & 5.12e-1 & 0.55 & 1.77e-2 & 1.52 &  & 1.37e-0 & 0.36 & 4.19e-2 & 1.32 \\
		$2^{-4}$ &  & 3.42e-1 & 0.58 & 6.03e-3 & 1.55 &  & 1.04e-0 & 0.40 & 1.67e-2 & 1.33 \\
		$2^{-5}$ &  & 2.23e-1 & 0.62 & 2.00e-3 & 1.59 &  & 7.56e-1 & 0.46 & 6.35e-3 & 1.39 \\
		$2^{-6}$ &  & 1.42e-1 & 0.65 & 6.49e-4 & 1.63 &  & 5.18e-1 & 0.55 & 2.26e-3 & 1.49 \\
		\bottomrule
	\end{tabular}
\end{table}

\begin{table}[ht]
	\caption{Convergence history with
		$\alpha=0.8$, $r=-0.8$, and $ h = 2^{-10}$ ($ \widetilde u $ is the numerical
		solution at $ \tau=2^{-17} $).
	}
	\label{tab:ex2-time}
	\small\setlength{\tabcolsep}{1.5pt}
	\begin{tabular}{cccccccccccc}
		\toprule
		\multicolumn{6}{c}{$\mathcal E_1$} & &
		\multicolumn{5}{c}{$\mathcal E_2$}  \\
		\cmidrule{1-6} \cmidrule{8-12}
		$\tau $ & & $ c=0 $ & Order & $ c=1 $ & Order & \phantom{aa}
		& $\tau $ & $ c=0 $ & Order & $ c=1 $ & Order \\
		\midrule
		$2^{-4}$ &  & 3.08e-1 & --   & 8.32e-1 & --   &  & $2^{-7}$  & 1.53e-2 & --   & 2.69e-2 & --   \\
		$2^{-5}$ &  & 2.55e-1 & 0.27 & 7.34e-1 & 0.18 &  & $2^{-8}$  & 1.05e-2 & 0.55 & 1.88e-2 & 0.52 \\
		$2^{-6}$ &  & 2.09e-1 & 0.29 & 6.50e-1 & 0.18 &  & $2^{-9}$  & 6.91e-3 & 0.60 & 1.31e-2 & 0.53 \\
		$2^{-7}$ &  & 1.69e-1 & 0.30 & 5.75e-1 & 0.18 &  & $2^{-10}$ & 4.47e-3 & 0.63 & 9.00e-2 & 0.54 \\
		$2^{-8}$ &  & 1.37e-1 & 0.31 & 5.06e-1 & 0.18 &  & $2^{-11}$ & 2.84e-3 & 0.65 & 6.17e-2 & 0.55 \\
		$2^{-9}$ &  & 1.10e-1 & 0.31 & 4.44e-1 & 0.19 &  & $2^{-12}$ & 1.78e-3 & 0.68 & 4.19e-2 & 0.56 \\
		\bottomrule
	\end{tabular}
\end{table}

\noindent{\bf Experiment 3.} This experiment verifies \cref{thm:conv_III}. Here we set $
\alpha = 0.8 $ and
\begin{alignat*}{2}
	u_0(x) & := 0, & \quad & 0 < x < 1, \\
	f(x,t) & := x^{-0.49} t^{-0.29}, & \quad & 0 < x < 1,\ 0 < t < T.
\end{alignat*}
\cref{thm:conv_III} implies that the convergence orders of $ \mathcal E_1 $ and $
\mathcal E_2 $ are $ \mathcal O(h+\tau^{0.6}) $ and $ \mathcal O(h^2+\tau) $,
respectively, which is confirmed by \cref{tab:ex3-space,tab:ex3-time}.

\newfloatcommand{capbtabbox}{table}[][\FBwidth]
\begin{table}[H]
	\begin{floatrow}
		\capbtabbox{
			\small\setlength{\tabcolsep}{1.3pt}
			\begin{tabular}{ccccccc}
				\toprule
				$h$ & \phantom{a} &  $\mathcal E_1$ & Order &  & $\mathcal E_2$ & Order \\
				\midrule
				$2^{-3}$ &  & 1.09e-2 & --   &  & 4.08e-3 & --   \\
				$2^{-4}$ &  & 5.87e-2 & 0.89 &  & 1.11e-3 & 1.88 \\
				$2^{-5}$ &  & 3.13e-2 & 0.91 &  & 2.98e-4 & 1.90 \\
				$2^{-6}$ &  & 1.66e-2 & 0.92 &  & 7.92e-5 & 1.91 \\
				$2^{-7}$ &  & 8.71e-3 & 0.93 &  & 2.09e-5 & 1.92 \\
				$2^{-8}$ &  & 4.55e-3 & 0.94 &  & 5.47e-6 & 1.93 \\
				\bottomrule
			\end{tabular}
		}{
			\caption{Convergence history with $ \tau = 2^{-15}$ ($ \widetilde u $ is the numerical solution at $ h=2^{-12} $).} \label{tab:ex3-space}
		}
		\capbtabbox{
			\small\setlength{\tabcolsep}{1.3pt}
			\begin{tabular}{ccccccc}
				\toprule
				$\tau$
				& \phantom{a} &  $\mathcal E_1$ & Order &  & $\mathcal E_2$ & Order \\
				\midrule
				$2^{-6}$  &  & 2.32e-2 & --   &  & 6.75e-3 & --   \\
				$2^{-7}$  &  & 1.52e-2 & 0.62 &  & 4.19e-3 & 0.69 \\
				$2^{-8}$  &  & 9.73e-3 & 0.64 &  & 2.47e-3 & 0.76 \\
				$2^{-9}$  &  & 6.22e-3 & 0.65 &  & 1.41e-3 & 0.81 \\
				$2^{-10}$ &  & 3.97e-3 & 0.65 &  & 7.81e-4 & 0.85 \\
				$2^{-11}$ &  & 2.54e-3 & 0.65 &  & 4.27e-4 & 0.87 \\
				\bottomrule
			\end{tabular}
		}{
			\caption{Convergence history with $ h = 2^{-10}$ ($ \widetilde u $ is the numerical solution at  $ \tau=2^{-17} $).} \label{tab:ex3-time}
		}
	\end{floatrow}
\end{table}

\appendix

\section{Properties of Fractional Calculus Operators}
\begin{lem}[\cite{Samko1993, Diethelm2010,Podlubny1998}]
  \label{lem:basic-frac}
  Let $ -\infty < a < b < \infty $. If $ 0 < \beta, \gamma < \infty $, then
  \[
    \I_{a+}^\beta \I_{a+}^\gamma = \I_{a+}^{\beta+\gamma}, \quad
    \I_{b-}^\beta \I_{b-}^\gamma = \I_{b-}^{\beta+\gamma},
  \]
  and
  \[
    \dual{\I_{a+}^\beta v,w}_{(a,b)} =
    \dual{v, \I_{b-}^\beta w}_{(a,b)}
  \]
  for all $ v, w \in L^2(a,b) $.
\end{lem}

\begin{lem}[\cite{Ervin2006}]
  \label{lem:coer}
  Assume that $ -\infty < a < b < \infty $ and $ 0 < \gamma < 1/2 $. If $ v \in
  H^\gamma(a,b) $, then
  \begin{align*}
    & \nm{\D_{a+}^\gamma v}_{L^2(a,b)} \leqslant \snm{v}_{H^\gamma(a,b)}, \\
    & \nm{\D_{b-}^\gamma v}_{L^2(a,b)} \leqslant \snm{v}_{H^\gamma(a,b)}, \\
    & \dual{\D_{a+}^\gamma v, \D_{b-}^\gamma v}_{(a,b)}
    = \cos(\gamma\pi) \snm{v}_{H^{\gamma}(a,b)}^2.
  \end{align*}
  Moreover, if $ v,w \in H^{\gamma}(a,b) $, then
  \begin{align*}
    & \dual{\D_{a+}^\gamma v, \D_{b-}^\gamma w}_{(a,b)} \leqslant
    \snm{v}_{H^\gamma(a,b)} \snm{w}_{H^\gamma(a,b)},\\
    & \dual{\D_{a+}^{2\gamma} v, w}_{H^\gamma(a,b)} =
    \dual{\D_{a+}^\gamma v, \D_{b-}^\gamma w}_{(a,b)} =
    \dual{\D_{b-}^{2\gamma} w, v}_{H^\gamma(a,b)}.
  \end{align*}
\end{lem}

\begin{lem}
  \label{lem:key}
  If $ 0 < \gamma < 1/2 $ and $ v \in L^2(0,1) $, then
  \begin{equation}
    \label{eq:key}
    C_1 \nm{\I_{0+}^{\gamma} v}_{L^2(0,1)}^2 \leqslant
    \left(
      \I_{0+}^{\gamma} v, \I_{T-}^{\gamma} v
    \right)_{L^2(0,1)}
    \leqslant C_2 \nm{\I_{0+}^\gamma v}_{L^2(0,1)}^2,
  \end{equation}
  where $ C_1 $ and $ C_2 $ are two positive constants that depend only on $
  \gamma $.
\end{lem}
\begin{proof}
  Extending $ v $ to $ \mathbb R \backslash (0,1) $ by zero, we define
  \begin{align*}
    w_{+}(t) &:= \frac1{ \Gamma(\gamma) }
    \int_{-\infty}^t (t-s)^{\gamma-1} v(s) \, \mathrm{d}s,
    \quad -\infty < t < \infty, \\
    w_{-}(t) &:= \frac1{ \Gamma(\gamma) }
    \int_t^{\infty} (s-t)^{\gamma-1} v(s) \, \mathrm{d}s,
    \quad -\infty < t < \infty.
  \end{align*}
  Since $ 0 < \gamma < 1/2 $, a routine calculation yields $ w_{+}, w_{-}
  \in L^2(\mathbb R) $, and \cite[Theorem~7.1]{Samko1993} implies that
  \[
    \begin{array}{ll}
      \mathcal Fw_{+}(\xi) = (\mathrm{i}\xi)^{-\gamma} \mathcal Fv(\xi),
      & -\infty < \xi < \infty, \\
      \mathcal Fw_{-}(\xi) = (-\mathrm{i}\xi)^{-\gamma} \mathcal Fv(\xi),
      & -\infty < \xi < \infty.
    \end{array}
  \]
  By the Plancherel Theorem and the same technique as that used to prove
  \cite[Lemma~2.4]{Ervin2006}, it follows that
  \begin{align*}
    {} &
    \left(
      \I_{0+}^{\gamma} v, \I_{1-}^{\gamma} v
    \right)_{L^2(0,1)} =
    (w_{+},w_{-})_{L^2(\mathbb R)} =
    (\mathcal Fw_{+}, \mathcal Fw_{-})_{L^2(\mathbb R)} \\
    ={} &
    \cos\big( \gamma\pi \big) \int_\mathbb R \snm\xi^{-2\gamma}
    \snm{\mathcal Fv(\xi)}^2 \, \mathrm{d}\xi \\
    ={} &
    \cos( \gamma\pi ) \nm{w_{+}}_{L^2(\mathbb R)}^2 =
    \cos( \gamma\pi ) \nm{w_{-}}_{L^2(\mathbb R)}^2.
  \end{align*}
  Therefore, by the Cauchy-Schwarz inequality, \cref{eq:key} follows from the
  following two estimates:
  \[
    \nm{ \I_{0+}^{\gamma} v }_{L^2(0,1)}
    \leqslant \nm{w_{+}}_{ L^2(\mathbb R) }, \quad
    \nm{ \I_{1-}^{\gamma} v }_{L^2(0,1)}
    \leqslant \nm{w_{-}}_{ L^2(\mathbb R) }.
  \]
\end{proof}

\begin{lem}
  \label{lem:regu}
  If $ \beta \in (0,1) \setminus \{0.5\} $ and $ 0 < \gamma < \infty $, then
  \begin{equation}
    \label{eq:regu-1}
    \nm{\I_{0+}^\gamma v}_{H^{\beta+\gamma}(0,1)} \leqslant
    C_{\beta,\gamma} \nm{v}_{H^\beta(0,1)}
  \end{equation}
  for all $ v \in H_0^\beta(0,1) $. Furthermore, if $ 0 < \gamma < 1/2 $ and $
  v \in H^{1-\gamma}(0,1) $ with $ v(0) = 0 $, then
  \begin{equation}
    \label{eq:regu-2}
    \nm{\I_{0+}^\gamma v}_{H^1(0,1)} \leqslant C_\gamma \nm{v}_{H^{1-\gamma}(0,1)}.
  \end{equation}
\end{lem}
\begin{proof}
  For the proof of \cref{eq:regu-1}, we refer the reader to \cite{Li2017A} (Lemma A.4). Let us prove
  \cref{eq:regu-2} as follows. Define $ \widetilde v := v - g $, where
  \[
    g(t) := t v(1), \quad 0 < t < 1.
  \]
  It is clear that $ \widetilde v \in H_0^{1-\gamma}(0,1) $, and hence
  \cref{eq:regu-1} implies
  \[
    \nm{ \I_{0+}^\gamma \widetilde v }_{H^1(0,1)} \leqslant
    C_\gamma \nm{\widetilde v}_{ H_0^{1-\gamma}(0,1) }.
  \]
  Therefore, from the evident estimate
  \[
    \nm{g}_{H^{1-\gamma}(0,1)} + \nm{\I_{0+}^\gamma g}_{H^1(0,1)}
    \leqslant C_\gamma \snm{v(1)},
  \]
  it follows that
  \begin{align*}
    \nm{\I_{0+}^\gamma v}_{H^1(0,1)}
    & \leqslant
    \nm{ \I_{0+}^\gamma \widetilde v }_{H^1(0,1)} +
    \nm{ \I_{0+}^\gamma g }_{H^1(0,1)} \\
    & \leqslant
    C_\gamma \nm{ \widetilde v }_{ H_0^{1-\gamma}(0,1) } +
    \nm{ \I_{0+}^\gamma g }_{H^1(0,1)} \\
    & \leqslant
    C_\gamma \left(\nm{v}_{ H^{1-\gamma}(0,1) } + \nm{g}_{ H^{1-\gamma}(0,1) } \right) +
    \nm{ \I_{0+}^\gamma g }_{H^1(0,1)} \\
    & \leqslant
    C_\gamma \left( \nm{v}_{ H^{1-\gamma}(0,1) } + \snm{v(1)} \right).
  \end{align*}
  As $ 0 < \gamma < 1/2 $ implies
  \[
    \nm{v}_{C[0,1]} \leqslant C_\gamma \nm{v}_{H^{1-\gamma}(0,1)},
  \]
  this indicates \cref{eq:regu-2} and thus proves the lemma.
\end{proof}

\begin{lem}
  \label{lem:xy}
  If $ 0 < \gamma < 1/2$ and $ v \in H^1(0,1) $, then
  \begin{equation}
    \label{eq:xy}
    C_1 \nm{v}_{ H^{1-\gamma}(0,1) } \leqslant
    \snm{v(0)} + \nm{ \I_{0+}^{\gamma} v' }_{L^2(0,1)} \leqslant
    C_2 \nm{v}_{H^{1-\gamma}(0,1)},
  \end{equation}
  where $ C_1 $ and $ C_2 $ are two positive constants that depend only on $
  \gamma $.
\end{lem}
\begin{proof}
  Since a simple calculation gives
  \[
    \D \I_{0+}^{\gamma} (v-v(0)) =
    \D \I_{0+}^{\gamma} \I_{0+} v' =
    \I_{0+}^{\gamma} v',
  \]
  using \cref{lem:regu} yields
  \begin{align*}
    {} &
    \nm{ \I_{0+}^{\gamma} v' }_{L^2(0,1)} \leqslant
    \nm{ \I_{0+}^{\gamma}(v-v(0)) }_{H^1(0,1)} \\
    \leqslant{} &
    C_\gamma \nm{v-v(0)}_{ H^{1-\gamma}(0,1) } \leqslant
    C_\gamma \left( \snm{v(0)} + \nm{v}_{ H^{1-\gamma}(0,1) } \right),
  \end{align*}
  which, together with the estimate
  \[
    \snm{v(0)} \leqslant C_\gamma \nm{v}_{H^{1-\gamma}(0,1)}
    \quad \text{(since $ 1-\gamma > 0.5 $)},
  \]
  indicates
  \[
    \snm{v(0)} + \nm{ \I_{0+}^{\gamma} v' }_{L^2(0,1)}
    \leqslant C_\gamma \nm{v}_{ H^{1-\gamma}(0,1) }.
  \]
  Conversely, by
  \[
    v = \I_{0+}^{ 1-\gamma } \I_{0+}^{ \gamma } v' + v(0),
  \]
  using \cref{lem:regu} again yields
  \[
    \nm{v}_{ H^{1-\gamma}(0,1) } \leqslant
    C_\gamma \left( \snm{v(0)} + \nm{ \I_{0+}^{\gamma}v' }_{L^2(0,1)} \right).
  \]
  This lemma is thus proved.
\end{proof}

%}}}

% \bibliographystyle{plain}
% \bibliography{fem_frac_diffu_final}

\end{document}